\documentclass[12pt, twoside]{article}
\usepackage{amsmath,amsthm,amssymb}
\usepackage{times}
\usepackage{enumerate}
\usepackage{bm}
\usepackage[all]{xy}
\usepackage{mathrsfs}
\usepackage{amscd}
\usepackage{color}
\def\titlerunning#1{\gdef\titrun{#1}}
\makeatletter
\def\author#1{\gdef\autrun{\def\and{\unskip, }#1}\gdef\@author{#1}}
\def\address#1{{\def\and{\\\hspace*{18pt}}\renewcommand{\thefootnote}{}%
\footnote {#1}}%
\markboth{\autrun}{\titrun}}
\makeatother
\def\email#1{e-mail: #1}

\newtheorem{theorem}{Theorem}[section]
\newtheorem{corollary}[theorem]{Corollary}
\newtheorem{lemma}[theorem]{Lemma}
\newtheorem{proposition}[theorem]{Proposition}
\theoremstyle{definition}
\newtheorem{definition}[theorem]{Definition}
\newtheorem{remark}[theorem]{Remark}
\numberwithin{equation}{section}

\xyoption{all}

\def \N {\mathbb{N}}
\def \C {\mathbb{C}}

\def \a {\alpha }
\def \b {\beta}

\def \de {\delta}
\def \De {\Delta}
\def \la {\lambda}
\def \La {\Lambda}
\def\w {\omega}
\def\Om{\Omega}

\def\pa{\partial}
\def\na {\nabla}
\def\Ga{\Gamma}

\begin{document}
\baselineskip=17pt

\titlerunning{On Yang-Mills connections on compact K\"{a}hler surfaces}
\title{On Yang-Mills connections on compact K\"{a}hler surfaces}
\author{Teng Huang}
\date{}
\maketitle
\address{T. Huang: School of Mathematical Sciences, University of Science and Technology of China; Key Laboratory of Wu Wen-Tsun Mathematics, Chinese Academy of Sciences, Hefei, Anhui 230026, PR China;  \email{htmath@ustc.edu.cn; htustc@gmail.com}}

\begin{abstract}
We extend an $L^{2}$-energy gap of Yang-Mills connections on principal $G$-bundles $P$ over a compact Riemannian manfold with a $good$ Riemannian metric \cite{Feehan}  to the case of a compact K\"{a}hler surface with a $generic$ K\"{a}hler metric $g$, which guarantees that all ASD connections on the principal bundle $P$ over $X$ are irreducible.  
\end{abstract}
\section{Introduction}
Let $G$ be a compact, semisimple Lie group and $P$ be a principal $G$-bundle over a closed, smooth, Riemannian manifold with Riemannian metric $g$. Suppose that $A$ is a connection on $P$ and its curvature denote by $F_{A}\in\Om^{2}(X)\otimes\mathfrak{g}_{P}$. Here $\Om^{k}:=\Om^{k}(T^{\ast}X)$ and $\mathfrak{g}_{P}$ is the real vector bundle associated to $P$ by the adjoint representation of $G$ on its Lie algebra $\mathfrak{g}$. We define the Yang-Mills energy function
by $$YM(A):=\int_{X}|F_{A}|^{2}dvol_{g},$$
the fiber metric defined through the Killing form on $\mathfrak{g}$, see \cite[Section 2]{Feehan}. Then energy functional $YM(A)$ is gauge-invariant and thus descends to a function on the quotient space $\mathcal{B}(P,g):=\mathcal{A}_{P}/\mathcal{G}_{P}$, of the affine space $\mathcal{A}_{P}$ of connections on $P$ moduli the gauge transformation. A connection $A$ is called Yang-Mills connection when it gives a critical point of the Yang-Mills functional, that is, it satisfies the Yang-Mills equation
$$d^{\ast}_{A}F_{A}=0.$$
From the Bianchi identity $d_{A}F_{A}=0$, a Yang-Mills connection is nothing but a connection whose curvature is harmonic with respect to the covariant exterior derivative $d_{A}$. 

Over a $4$-dimensional Riemannnian manifold, $F_{A}$ is decomposed into its self-dual and anti-self-dual components,
$$ F_{A}=F^{+}_{A}+F^{-}_{A}$$
where $F^{\pm}_{A}$ denotes the projection onto the $\pm1$ eigenspace of the Hodge star operator. A connection is called self-dual (respectively anti-self-dual) if $F_{A}=F^{+}_{A}$ (respectively $F_{A}=F^{-}_{A}$). A connection is called an instanton if it is either self-dual or anti-self-dual. On compact oriented $4$-manifolds, an instanton is always an absolute minimizer of the Yang-Mills energy \cite{Taubes1982,Taubes1984}. Not all Yang-Mills connections are instantons, in \cite{SS,SSU}, the authors given some examples for the $SU(2)$ Yang-Mills connections on $S^{4}$ which are neither self-dual nor anti-self-dual. 

It is a natural question whether or not there is a positive uniform gap between the energy $YM(A)$ of points $[A]$ in the stratum $$M(P,g):=\{[A]\in\mathcal{B}(P,g): F_{A}^{+}=0\},$$ of absolute minimal of $YM(A)$ on $\mathcal{B}(P,g)$ and energies of points in the strata in $\mathcal{B}(P,g)$ of non-minimal critical points. 

In \cite{BL,BLS}, Bourguignon-Lawson proved that if $A$ is a Yang-Mills on a principal  $G$-bundle  over $S^{4}$ with its standard round metric of radius one such that $\|F_{A}^{+}\|_{L^{\infty}(X)}<\sqrt{3}$, then $A$ is anti-self-dual. The result was significantly improved by Min-Oo \cite{Mi} and Parker \cite{Parker}, by replacing the preceding $L^{\infty}$ condition with an $L^{2}$-energy condition, $\|F_{A}^{+}\|_{L^{2}(X)}\leq\varepsilon$, where $\varepsilon=\varepsilon(g)$ is a small enough constant and by assume $X$ to be a closed, smooth, four-dimensional manifold endowed with a $positive$ Riemannian metric $g$. In \cite{Feehan}, Feehan extend the  $L^{2}$-energy gap result  from the case of positive Riemannian metrics \cite{Mi,Parker} to the more general case of $good$ Riemannian metrics.  The key step in the proof of Feehan's \cite{Feehan} Theorem 1 is to derive an uniform positive lower bound for the lower eigenvalue of the operator $d_{A}^{+}d_{A}^{+,\ast}$ with respect to the connection $A$, the curvature $F_{A}$ obeying $\|F_{A}^{+}\|_{L^{2}(X)}\leq\varepsilon$ for a suitable small constant $\varepsilon=\varepsilon(g)$. In \cite{GKS}, the authors showed a sharp, conformally invariant improvement of these gap theorems which is nontrivial when the Yamabe invariant $Y([g])$ of $(X,g)$ is positive.

Denote by $\mathcal{A}_{YM}$ the space of Yang-Mills connections and $\mathcal{A}_{HYM}$ the space of connections whose curvature satisfies $$\sqrt{-1}\La_{\w}F_{A}=\la Id,$$
where $\la=\frac{2\pi deg P}{rank(P)vol(X)}$. There spaces are gauge invariant with respect to the group $\mathcal{G}_{P}$ of gauge transformations. In \cite{Huang}, the author proved that an open subset $$W=\{[A]:\|\sqrt{-1}\La_{\w}F_{A}-\la Id\|_{L^{2}(X)}<\de\}$$ in the orbit space $\mathcal{A}_{P}/\mathcal{G}_{P}$ of connections with property $\mathcal{A}_{HYM}/\mathcal{G}_{P}=W\cap\mathcal{A}_{YM}/\mathcal{G}_{P}$ under the scalar curvature $S$ of the metric is positive. 

Now if we suppose the base manifold $X$ is a K\"{a}hler surface, $P$ is a principal $SU(N)$-bundle over $X$. An ASD connection $A$ on $P$ naturally induces the Yang-Mills complex
$$\Om^{0}(X,\mathfrak{g}_{P})\xrightarrow{d^{0}=d_{A}}\Om^{1}(X,\mathfrak{g}_{P})\xrightarrow{d^{1}=d^{+}_{A}}{\Om^{2,+}(X,\mathfrak{g}_{P})}.$$
The $i$-th cohomology group $H^{i}_{A}:=Ker d^{i}/ Im d^{i-1}$ of this complex if finite dimensional and the index $d=h^{0}-h^{1}+h^{2}$ ($h^{i}=dim H^{i}_{A}$ )is given by
$c(G)\kappa(P)-dim G(1-b_{1}+b^{+})$. Here $c(G)$ is a normalising constan t, $\kappa(P)$ is a characteristic number of $P$ obtained by evaluating a $4$-dimensional characteristic class on the fundamental cycle $[X]$, $b_{1}$ is the first Betti number of $X$ and $b^{+}$ is the rank of a maximal positive subspace for the intersection form on $H^{2}(X)$. $H^{0}_{A}$ is the Lie algebra of the stabilizer $\Ga_{A}$, the group of gauge transformation of $P$ fixing by $A$. From \cite[Proposition 2.3]{Itoh}  or \cite[Chapter IV]{Friedman-Morgan}, the second cohomology $H^{2}_{A}$ is $\mathbf{R}$-isomorphic to $H^{0}_{A}\oplus\mathbf{H}$, where $\mathbf{H}:=\ker(\bar{\pa}_{A}^{\ast})|_{\Om^{0,2}(X,\mathfrak{g}_{P}^{\C})}$
It's difficult to addition certain mild conditions to ensure  $\mathbf{H}$ and $H^{0}_{A}$ vanish at some time. The K\"{a}hler metric $g$ often could not be  $good$. But one can see that $H^{0}_{A}=0$ is equivalent to the connection $A$ is irreducible.   

In \cite{Feehan}, the author shown that if the close $4$-manifold $X$ admits a $good$ metric $g$, then the connection $A\in\mathcal{B}_{\varepsilon}(P,g):=\{[A]\in\mathcal{B}(P,g):\|F_{A}^{+}\|_{L^{2}(X)}\leq\varepsilon \}$ such that the last eigenvalue of $d_{A}^{+}d_{A}^{+,\ast}$ on $L^{2}(\Om^{2,+}(X,\mathfrak{g}_{P}))$ has a lower positive bound, where $\varepsilon=\varepsilon(g)$ is  a suitable small positive constant. The purpose of this article is to introduce the definition of  $strong$ irreducible connection $A$ which only guarantees that $\la(A)$ has a low positive bound, See Definition \ref{D3}.

If the Riemannian metric is $good$, there is a well known gluing theorem for anti-self-dual connection which due to Taubes \cite{Taubes1982}. Following the idea of Taubes', if we suppose the connection $A\in\mathcal{A}_{P}$ which obeys $\|F_{A}^{+}\|_{L^{2}(X)}\leq \varepsilon$ for a suitable  small positive constant  and $\la(A)\geq\la_{0}>0$, then we could deform the connection $A$ to an other connection $A_{\infty}$ which satisfies $\La_{\w}F_{A_{\infty}}=0$, see Corollary \ref{C4}. The connection $A_{\infty}$ may be not an ASD connection, but the $(0,2)$-part $F_{A_{\infty}}^{0,2}$ of the curvature $F_{A_{\infty}}$ could estimated by $\La_{\w}F_{A}$. Following the priori estimate in Theorem \ref{T6} and the vanishing Theorem \ref{T2}, we have
\begin{theorem}\label{T1}
Let $X$ be a compact, simply-connected, K\"{a}hler surface with a K\"{a}hler metric $g$, $P$ be a principal $G$-bundle with $G$ being $SU(2)$ or $SO(3)$. Suppose that the connections $A\in M(P,g)$ are $strong$ irreducible ASD connections in the sense of Definition \ref{D3}, then there is a positive constant $\varepsilon=\varepsilon(g,P)$ with following significance. If the Yang-Mills connection $A$ on $P$ such that 
\begin{equation*}
\|F_{A}^{+}\|_{L^{2}(X)}\leq\varepsilon,
\end{equation*}
then $A$ is anti-self-dual with respect to $g$, i.e., $F_{A}^{+}=0$.
\end{theorem}
We may assume that any connection $A\in M(P,g)$ is $strong$ irreducible ASD connection in the sense of Definition \ref{D3} if the conditions in Theorem \ref{T4} are obeyed.
\begin{corollary}
Let $X$ be a compact, simply-connected, K\"{a}hler surface with a K\"{a}hler metric $g$, that is $generic$ in the sense of Definition \ref{D1}, $P$ be a $SO(3)$-bundle over $X$. Suppose that the second Stiefel-Whitney class, $\w_{2}(P)\in H^{2}(X;\mathbb{Z}/2\mathbb{Z})$, is non-trivial, then there is a positive constant $\varepsilon=\varepsilon(g,P)$ with following significance. If the curvature $F_{A}$ of a Yang-Mills connection $A$ on $P$ obeying
$$\|F_{A}^{+}\|_{L^{2}(X)}\leq\varepsilon,$$
then $A$ is anti-self-dual with respect to $g$, i.e., $F_{A}^{+}=0$.
\end{corollary}
This paper is organised as follows. In Section 2, we first establish our notation and recall basic definitions in gauge theory over K\"{a}hler manifolds required for the remainder of this article. Following the idea on \cite{Feehan}, we prove that  the least eigenvalue, $\la(A)$, of $d^{\ast}_{A}d_{A}$ has a positive lower bound $\la_{0}=\la_{0}(g,P)$ that is uniform with respect to $[A]\in\mathcal{B}(g,P)$ obeying $\|F^{+}_{A}\|_{L^{2}(X)}\leq\varepsilon$, for a small enough $\varepsilon=\varepsilon(g,P)\in (0, 1]$ and under the given sets of conditions on $g, G$. In Section 3, using the similar way of construct of ASD connection by Taubes \cite{Taubes1982}, we obtain that an approximate ASD connection $A\in\mathcal{A}_{P}$ could deform into an other approximate ASD connection $A_{\infty}$ which satisfies $\La_{\w}F_{A_{\infty}}=0$. Thus we can prove that if a Yang-Mills connection obeying $\|F^{+}_{A}\|_{L^{2}(X)}\leq\varepsilon$, then the curvature is harmonic and $\La_{\w}F_{A}=0$. In Section 4, we establish a vanishing theorem (Theorem \ref{T2}) on the space of $\tilde{\Om}^{0,2}(X,\mathfrak{g}_{P}^{\C})$ Equ. (\ref{E3.10}). Thus we also can prove that  the least eigenvalue, $\mu(A)$, of $\bar{\pa}_{A}\bar{\pa}^{\ast}_{A}$ on space $\tilde{\Om}^{0,2}(X,\mathfrak{g}_{P}^{\C})$ has a positive lower bound $\mu_{0}=\mu_{0}(g,P)$ that is uniform with respect to $[A]\in\mathcal{B}(g,P)$ obeying $\|F^{+}_{A}\|_{L^{2}(X)}\leq\varepsilon$. Combining the curvature is harmonic, then we complete the proof of Theorem \ref{T1}.
\section{Preliminaries}
\subsection{ Weitzenb\"{o}ck formula}
Let $X$ be a K\"{a}hler surface with K\"{a}hler form $\w$ and $P$ be a principal $G$-bundle over $X$. For any connection $A$ on $P$ we have the covariant exterior derivatives $d_{A}:\Om^{k}(X, \mathfrak{g}_{P})\rightarrow\Om^{k+1}(X,\mathfrak{g}_{P})$. Like the canonical splitting the exterior derivatives $d=\pa+\bar{\pa}$, decomposes over $X$ into $d_{A}=\pa_{A}+\bar{\pa}_{A}$. We denote also by $\Om^{p,q}(X,\mathfrak{g}_{P}^{\C})$ the space of $C^{\infty}$-$(p,q)$ forms on $\mathfrak{g}_{P}^{\C}:=\mathfrak{g}_{P}\otimes\C$. Denote by $L_{\w}$ the operator of exterior multiplication by the K\"{a}hler form $\w$:
$$L_{\w}\a=\w\wedge\a, \a\in\Om^{p,q}(X,\mathfrak{g}_{P}^{\C}),$$
and, as usual, let $\La_{\w}$ denote its pointwise adjoint, i.e.,
$$\langle\La_{\w}\a,\b\rangle=\langle\a,L_{\w}\b\rangle.$$
Then it is well known that $\La_{\w}=\ast^{-1}\circ L_{\w}\circ \ast$. We could decompose the curvature, $F_{A}$, as
$$F_{A}=F^{2,0}_{A}+F^{1,1}_{A0}+\frac{1}{2}\La_{\w}{F}_{A}\otimes\w+F^{0,2}_{A},$$
where  $F^{1,1}_{A0}=F^{1,1}_{A}-\frac{1}{2}\La_{\w}F_{A}\otimes\w$.\ We can write Yang-Mills functional as
\begin{equation}\nonumber
\begin{split}
YM(A)&=4\|F^{0,2}_{A}\|^{2}+\|\La_{\w} F_{A}\|^{2}+\int_{X}tr(F_{A}\wedge F_{A}).\\
\end{split}
\end{equation}
The energy functional $\|\La_{\w}F_{A}\|^{2}$ plays an important role in the study of Hermitian-Einstein connections, see \cite{DK,UY}. If the connection $A$ is an ASD connection, the Yang-Mills functional is minimum. We recall some identities on Yang-Mills connection over K\"{a}hler surface, see \cite[Proposition 3.1]{Itoh}  or \cite[Proposition 2.1]{Huang} .
\begin{proposition}\label{P1}
Let $A$ be a Yang-Mills connection on a principal $G$-bundle $P$ over a K\"{a}hler surface $X$, then we have following identities:
\begin{equation*}
\begin{split}
&(1)\ 2\bar{\pa}^{\ast}_{A}F^{0,2}_{A}=\sqrt{-1}\bar{\pa}_{A}\La_{\w}F_{A},\\
&(2)\ 2\pa_{A}^{\ast}F^{2,0}_{A}=-\sqrt{-1}\pa_{A}\La_{\w}F_{A}.\\
\end{split}
\end{equation*}
\end{proposition}
We define a Hermitian inner product $\langle\cdot,\cdot\rangle$ on $\Om^{p,q}(X,\mathfrak{g}_{P}^{\C})$ by
$$\langle\a,\b\rangle_{L^{2}(X)}=\int_{X}\langle\a,\b\rangle(x)dvol_{g},$$
$$\langle\a,\b\rangle(x)dvol_{g}=\langle\a\wedge\ast\bar{\b} \rangle,$$
where $\ast$ is the $\C$-linearly extend Hodge operator over complex forms and $\bar{}$ is the conjugation on the bundle
$\mathfrak{g}_{P}^{\C}$-forms which is defined naturally. One also can see \cite[Page 99]{Itoh2} or \cite{Huy}. We recall a Weitzenb\"{o}ck formula for Lie algebra-valued $(0,2)$-forms, see \cite[Proposition 2.3]{Itoh}, a self-dual operator denote by $\De_{\bar{\pa}_{A}}=\bar{\pa}_{A}\bar{\pa}^{\ast}_{A}+\bar{\pa}_{A}^{\ast}\bar{\pa}_{A}$.
\begin{proposition}\label{P3}
Let $X$ be a smooth K\"{a}hler surface with a K\"{a}hler metric $g$, $A$ be a connection on a principal $G$-bundle $P$ over $X$. For any $\phi\in\Om^{0,2}(X,\mathfrak{g}_{P}^{\C})$,
\begin{equation}\label{12}
\De_{\bar{\pa}_{A}}\phi=\na^{\ast}_{A}\na_{A}\phi+[\sqrt{-1}\La_{\w}F_{A},\phi]+2S\phi
\end{equation}
where $S$ is the scalar curvature of the metric $g$.
\end{proposition}
Combining Weitzenb\"{o}ck formula on Proposition \ref{P3} with the identities on Proposition \ref{P1}, we have an identity for Yang-Mills connection.
\begin{proposition}\label{P5}
Let $X$ be a smooth K\"{a}hler surface with a K\"{a}hler metric $g$, $A$ be a Yang-Mills connection on a principal $G$-bundle $P$ over $X$. Then we have
\begin{equation}\label{E14}
\na_{A}^{\ast}\na_{A}F^{0,2}_{A}+\frac{3}{2}[\sqrt{-1}\La_{\w}F_{A},F^{0,2}_{A}]+2SF^{0,2}_{A}=0
\end{equation}
Furthermore, if $X$ is compact,
\begin{equation}\label{E12}
\frac{3}{4}\|\bar{\pa}_{A}\La_{\w}F_{A}\|^{2}_{L^{2}(X)}=\|\na_{A}F^{0,2}_{A}\|^{2}_{L^{2}(X)}+\int_{X}2S|F_{A}^{0,2}|^{2}dvol_{g}.
\end{equation}
\end{proposition}
\begin{proof}
Following Proposition \ref{P1}, we obtain that
\begin{equation*}
\begin{split}
\|\bar{\pa}_{A}\La_{\w}F_{A}\|_{L^{2}(X)}^{2}&=\langle\bar{\pa}^{\ast}_{A}\bar{\pa}_{A}\La_{\w}F_{A},\La_{\w}F_{A}\rangle_{L^{2}(X)}\\&=-\langle2\sqrt{-1}\ast [F^{0,2}_{A},\ast F_{A}^{0,2}],\La_{\w}F_{A}\rangle_{L^{2}(X)}\\
&=-\langle2\sqrt{-1} [\La_{\w}F_{A}, F^{0,2}_{A}], F_{A}^{0,2}\rangle_{L^{2}(X)}\\
\end{split}
\end{equation*}
The Weizenb\"{o}ck formula for $F^{0,2}_{A}$ yields
$$\na_{A}^{\ast}\na_{A}F_{A}^{0,2}+2SF_{A}^{0,2}+\frac{3}{2}[\sqrt{-1}\La_{\w}F_{A},F_{A}^{0,2}]=0.$$
If $X$ is closed, taking  the $L^{2}$-inner product of above identity with $F_{A}^{0,2}$ and integrating by parts, we then obtain (\ref{E12}).
\end{proof}

\subsection{Irreducible connections}
In this section, we first recall a definition of irreducible connection on a principal $G$-bundle $P$, where $G$ being a compact, semisimple Lie group. Given a connection $A$ on a principal $G$-bundle $P$ over $X$. We can define the stabilizer $\Ga_{A}$ of $A$ in the gauge group $\mathcal{G}_{P}$ by
$$\Ga_{A}:=\{g\in\mathcal{G}_{P}|g(A)=A\},$$
one also can see \cite{DK}. A connection $A$ called reducible if the connection $A$ whose stabilizer $\Ga_{A}$ is larger than the centre $C(G)$ of $G$. Otherwise, the connections are irreducible, they satisfy $\Ga_{A}\cong C(G)$. It's easy to see that a connection $A$ is irreducible when it admits no nontrivial covariantly constant Lie algebra-value $0$-form, i.e., $\ker d_{A}|_{\Om^{0}(X,\mathfrak{g}_{P})}=0$. Taubes introduced the number in \cite[Equation 6.3]{Taubes1988} to measure the irreducibility of $A$. We can defined the least eigenvalue $\la(A)$ of $d^{\ast}_{A}d_{A}$ as follow.
\begin{definition}\label{D2}
The least eigenvalue of $d_{A}^{\ast}d_{A}$ on $L^{2}$ section of $\Ga(\mathfrak{g}_{P})$ is 
\begin{equation}\label{E3}
\la(A):=\inf_{v\in\Ga(\mathfrak{g}_{P})\backslash\{0\}}\frac{\|d_{A}v\|^{2}}{\|v\|^{2}}.
 \end{equation}
\end{definition}
It is easy to see that the function $\la(A)$ depends only on  the connection $A$. We introduce the definition of $strong$ irreducible connection on a principal $G$-bundle $P$. 
\begin{definition}\label{D3}
We call $A$ a smooth $strong$ irreducible connection on $G$-bundle $P$ over a smooth $n$-dimensional, Riemannian manifold $X$, ($n\geq2$), if the least eigenvalue of $d_{A}^{\ast}d_{A}$ on $L^{2}$ section of $\Ga(\mathfrak{g}_{P})$ has a positive lower bound, i.e, there is a constant $\la_{0}=\la_{0}(P,g)\in(0,\infty)$ such that $\la(A)\geq\la_{0}$.
\end{definition}

The Sobolev norms $L^{p}_{k,A}$,\ where $1\leq p<\infty$ and $k$ is an integer, with respect to the connections defined as:
\begin{equation}\nonumber
\|u\|_{L^{p}_{k,A}(X)}:=\big{(}\sum_{j=0}^{k}\int_{X}|\na^{j}_{A}u|^{p}dvol_{g}\big{)}^{1/p},
\end{equation}
where $\na_{A}$ is the covariant derivative induced by the connection $A$ on $P$ and the Levi-Civita connection defined by the Riemannian metric $g$ on $T^{\ast}X$ and $\na^{j}_{A}:=\na_{A}\circ\ldots\circ\na_{A}$ (repeated $j$ times for $j\geq0$). 
\begin{remark}\label{R2.5}
Let $A$ be a irreducible connection on the principal $G$-bundle over a compact manifold $X$, i.e.,  $\ker d_{A}|_{\Om^{0}(X,\mathfrak{g}_{P})}=0$. Then we can assume that $\ker d_{A}|_{\Om^{0}(X,\mathfrak{g}_{P}^{\C})}=0$, where $\mathfrak{g}_{P}^{\C}:=\mathfrak{g}_{P}\otimes\C$. We denote $s$ by a section of $\Ga(\mathfrak{g}_{P}^{\C})$, i.e., $s$ can be seen as a function over $X$ which takes value in the Lie algebra $\mathfrak{g}\otimes\C$.  Here $\mathfrak{g}\otimes\C$ is the complexification of Lie algebra $\mathfrak{g}$, See \cite[Pages 11--12]{Samelson} . Thus in a local coordinate, there exists two $s_{1},s_{2}\in\mathfrak{g}$ such that $s=s_{1}+\rm{i}s_{2}$. Then  we have two identities,
$|\na_{A}s|^{2}=|\na_{A}s_{1}|^{2}+|\na_{A}s_{2}|^{2}$ and $|s|^{2}=|s_{1}|^{2}+|s_{2}|^{2}$. Since $A$ is irreducible,  $\|\na_{A}s\|^{2}_{L^{2}(X)}\geq\la\|s\|^{2}_{L^{2}(X)}$.
\end{remark}
By the similar method of the proof of \cite[Proposition A.3]{Feehan} or \cite[Lemma 7.1.24]{DK}, the least eigenvalue $\la(A)$ of $d^{\ast}_{A}d_{A}$  with respect to connection $A$ is continuous in $L^{p}_{loc}$-topology for $2\leq p<4$. Due to a result of Sedlacek \cite[Theorem 4.3]{Sedlacek}, we then have
\begin{proposition}\label{P2.6}
Let $G$ be a compact,semisimple Lie group and P be a principal $G$-bundle over a closed, smooth, oriented, four-dimensional Riemannian manifold $X$ with a Riemannian metric $g$. If $\{A_{i}\}_{i\in\mathbb{N}}$ is a sequence $C^{\infty}$ connection on $P$ and the curvatures $\|F^{+}_{A_{i}}\|_{L^{2}(X)}$ converge to zero as $i\rightarrow\infty$, then there are a finite set of points, $\Sigma=\{x_{1},\ldots,x_{L}\}\subset X$ and a subsequence $\Xi\subset\N$, we also denote by $\{A_{i}\}$, a sequence gauge transformatin $\{g_{i}\}_{i\in\Xi}$ such that, $g_{i}(A_{i})\rightarrow A_{\infty}$, a anti-self-dual connection on $P\upharpoonright_{X\backslash\Sigma}$ in the $L^{2}_{1}$-topology over $X\backslash\Sigma$. Furthermore, we have
$$\lim_{i\rightarrow\infty}\la(A_{i})=\la(A_{\infty}),$$
where $\la(A)$ is as in Definition \ref{D2}.
\end{proposition} 
\begin{theorem}
Let $X$ be a closed, Riemannian 4-manifold, $P$ be a principal $G$-bundle with $G$ being compact, semisimple Lie group. Suppose that the ASD connections $A\in M(P,g)$ are $strong$ irreducible connections in the sense of Definition \ref{D3}, then there are positive constants $\la_{0}$ and $\varepsilon$ such that
\begin{equation*}
\begin{split}
&\la(A)\geq\la_{0},\ \forall [A]\in M(P,g),\\
&\la(A)\geq\frac{\la_{0}}{2},\ \forall [A]\in\mathcal{B}_{\varepsilon}(P,g).\\
\end{split}
\end{equation*}
\end{theorem}
\begin{proof}
By the definition of $strong$ irreducible connection, there is a positive constant $\la_{0}=\la_{0}(P,g)$ such that $\la(A)\geq\la_{0}$, $\forall [A]\in M(P,g)$.

Suppose that the constant $\varepsilon\in(0,1]$ does not exist. We may then choose a sequence $\{A_{i}\}_{i\in\N}$ of connection on $P$ such that $\|F^{+}_{A_{i}}\|_{L^{2}(X)}\rightarrow0$ and $\la(A_{i})\rightarrow0$ as $i\rightarrow\infty$. Since $\lim_{i\rightarrow\infty}\la(A_{i})=\la(A_{\infty})$ and $\la(A_{\infty})\geq\la_{0}$, then it is contradict to our initial assumption regarding the sequence $\{A_{i}\}_{i\in\N}$. In particular, the preceding argument shows that the desired constant $\varepsilon$ exists.
\end{proof}
Friedman and Morgan  introduced the $generic$ K\"{a}hler metric which ensures the connections on the compactification of moduli space of ASD connections on $E$, $\bar{M}(P,g)$,  are irreducible, See \cite[Chapter IV]{Friedman-Morgan}. The authors also given an example on a K\"{a}hler surface which admits a $generic$ K\"{a}hler metric. Fix an algebraic surface $S$ and an ample line bundle $L$ on $S$. For every integer $c$ we have defined the moduli space $\mathcal{M}_{c}(S,L)$ of $L$-stable rank two holomorphic vector bundles $V$ on $S$ and such that $c_{1}(V)=0$, $c_{2}(V)=c$. Next let us determine when a Hodge metric with K\"{a}hler form $\w$ admits reducible ASD connections. Corresponding to such a connection is an associated ASD harmonic $1$-form $\a$, well-defined up to $\pm1$, representing an integral cohomology class, which by the description of $\Om^{2,-}(X)$ is of type $(1,1)$ and orthogonal to $\w$. Thus Friedman-Morgan proved for an integer $c>0$, there is an open dense subset $\mathcal{D}$ of the cone of ample divisors on $S$ such that if $g$ is a Hodge metric whose K\"{a}hler form lies in $\mathcal{D}$, $g$ is a $generic$ metric in the sense of Definition \ref{D1}, See \cite[Chapter IV, Proposition 4.8]{Friedman-Morgan}. 
\begin{definition}\label{D1}
Let  $X$ be a compact K\"{a}hler surface, $P$ be a principal $G$-bundle over $X$ with $c_{2}(P)=c$. We say that a K\"{a}hler metric $g$ on $X$ is $generic$ if for every $G$-bundle $\tilde{P}$ over $X$ with $0<c_{2}(\tilde{P})\leq c$, there are no reducible ASD connections on $\tilde{P}$. 
\end{definition}

For a compact K\"{a}hler surface $X$ we have a moduli space of ASD connections $M(P,g)$. The Donaldson-Uhlenbeck compactification $\bar{M}(P,g)$ of $M(P,g)$ contained in the disjoint union
\begin{equation}\nonumber
\bar{M}(P,g)\subset\cup(M(P_{l},g)\times Sym^{l}(X)),
\end{equation}
Following \cite[Theorem 4.4.3]{DK}, the space $\bar{M}(P,g)$ is compact. 
\begin{lemma}
Let $X$ be a compact, simply-connected, K\"{a}hler surface with a $generic$ K\"{a}hler metric, $P$ be a $SO(3)$-bundle over $X$. If the second Stiefel-Whitney classes $w_{2}\in H^{2}(X;\mathbb{Z}/2\mathbb{Z})$ is non-trivial, then $\la(A)>0$ for any $[A]\in\bar{M}(P,g)$. 	
\end{lemma}
\begin{proof}
The only reducible anti-self-dual connection on a principal $SO(3)$-bundle over $X$ is the product connection on the product bundle $P=X\times SO(3)$ by \cite[Corollary 4.3.15]{DK} and the latter possibility is excluded by our hypothesis in this case that $w_{2}(P)\neq 0$. Then the conclusion is a consequence of the definition of $generic$ K\"{a}hler metric.
\end{proof} 
Following Feehan's idea, combining $\bar{M}(P,g)$ is compact and $\la(A)$ is continuous under the Uhlenbeck topology for the connection $[A]\in\bar{M}(P,g)$, then we have
\begin{theorem}\label{T4}
Let $X$ be a compact, simply-connected, K\"{a}hler surface with a $generic$ K\"{a}hler metric, $P$ be a $SO(3)$-bundle over $X$. If the second Stiefel-Whitney classes $w_{2}(P)\in H^{2}(X;\mathbb{Z}/2\mathbb{Z})$ is non-trivial,  then there are positive constants $\mu_{0}=\mu_{0}(P,g)$ and $\varepsilon=\varepsilon(P,g)$ such that
\begin{equation*}
\begin{split}
&\la(A)\geq\la_{0},\ \forall [A]\in M(P,g),\\
&\la(A)\geq\frac{\la_{0}}{2},\ \forall [A]\in\mathcal{B}_{\varepsilon}(P,g).\\
\end{split}
\end{equation*}
\end{theorem}
\subsection{An estimate on Yang-Mills connections}\label{S1}
\begin{lemma}\label{L1}
Let $X$ be a closed, smooth, four-manifold with a Riemannian metric $g$, $A$ be a connection on a principal $G$-bundle $P$ over $X$ with $G$ being a compact, semisimple Lie group. There are positive constants $\la=\la(g,P)$ and  $C=C(\la,g,P)$ with following significance. If  $\la(A)\geq\la$, where $\la(A)$ is as in (\ref{E3}), then for any section $v$ on $\Ga(\mathfrak{g}_{P})$, 
\begin{equation}\label{E6}
\|v\|_{L^{2}_{2}(X)}\leq C\|\na_{A}^{\ast}\na_{A}v\|_{L^{2}(X)}.
\end{equation}
\end{lemma}
\begin{proof}
Since $\na_{A}^{\ast}\na_{A}$ is a elliptic operator of degree $2$, from a priori estimate of elliptic operator, for $p\geq1$ and $k\geq 0$, we have 
\begin{equation}\label{E7}
\|v\|_{L^{p}_{k+2}(X)}\leq c\|\na_{A}^{\ast}\na_{A}v\|_{L^{p}_{k}(X)}+\|v\|_{L^{p}(X)}
\end{equation}
We take $k=0$ and $p=2$, 
$$\|v\|_{L^{2}_{2}(X)}\leq c\|\na_{A}^{\ast}\na_{A}v\|_{L^{2}(X)}+\|v\|_{L^{2}(X)}.$$
By the definition of $\la(A)$, we also have
$$\|v\|_{L^{2}(X)}\leq \la^{-1}\|\na_{A}^{\ast}\na_{A}v\|_{L^{2}(X)}.$$
Combining the preceding inequalities yields (\ref{E6}).
\end{proof}
We construct a priori on the Yang-Mills connection under the condition of $\La_{\w}F_{A}$ is sufficiently small in $L^{2}$-norm.
\begin{proposition}\label{P4}
Let $X$ be a compact K\"{a}hler  surface, $A$ be  a Yang-Mills connection on a principal $G$-bundle $P$ over $X$ with $G$ being a compact, semisimple Lie group, $p\in(4,\infty)$. Let $q\in(4/3,2)$ be defined by $1/q=1/2+1/p$. There are positive constants $\la=\la(g,P)$ and  $\varepsilon=\varepsilon(\la,g,P)$ with following significance. If the curvature $F_{A}$ of the connection $A$ on $P$ obeying
$$\|\La_{\w}F_{A}\|_{L^{2}(X)}\leq\varepsilon,$$
and  $\la(A)\geq\la$, where $\la(A)$ is as in (\ref{E3}), then
\begin{equation}\label{E4}
\|F_{A}^{0,2}\|_{L^{p}(X)}\leq C\|F_{A}^{0,2}\|_{L^{q}(X)},
\end{equation}
where $C=C(\la,g,P,p)$ is a positive constant.
\end{proposition}
\begin{proof}
We apply the estimate (\ref{E7}) and Sobolev embedding $L^{q}_{2}\hookrightarrow L^{p}$, then
$$\|F_{A}^{0,2}\|_{L^{p}(X)}\leq c\|F_{A}^{0,2}\|_{L^{q}_{2}(X)}\leq  \|\na_{A}^{\ast}\na_{A}F_{A}^{0,2}\|_{L^{q}(X)}+\|F^{0,2}_{A}\|_{L^{q}(X)}.$$
By using the Weizenb\"{o}ck formula of $F^{0,2}_{A}$, see equation (\ref{E14}), we then observe that
\begin{equation*}
\begin{split}
\|\na_{A}^{\ast}\na_{A}F_{A}^{0,2}\|_{L^{q}(X)}&\leq c\|F^{0,2}_{A}\|_{L^{q}(X)}+\|\{F_{A}^{0,2},\La_{\w}F_{A}\}\|_{L^{q}(X)}\\
&\leq c\|F^{0,2}_{A}\|_{L^{q}(X)}+c\|\La_{\w}F_{A}\|_{L^{2}(X)}\|F_{A}^{0,2}\|_{L^{p}(X)}.\\
\end{split}
\end{equation*}
where $c=c(g,G,p)$ is a positive constant.
Combining the preceding inequalities  gives
$$\|F_{A}^{0,2}\|_{L^{p}(X)}\leq c\|F^{0,2}_{A}\|_{L^{q}(X)}+c\|\La_{\w}F_{A}\|_{L^{2}(X)}\|F_{A}^{0,2}\|_{L^{p}(X)}+\|F^{0,2}_{A}\|_{L^{q}(X)}.$$
where $c=c(g,G,p)$ is a positive constant. Provide $c\|\La_{\w}F_{A}\|_{L^{2}(X)}\leq1/2$, rearrangement gives (\ref{E4}).
\end{proof} 
\section{Yang-Mills connection on K\"{a}hler surface}
\subsection{Approximate Hermitian-Yang-Mills connections}
In this section we will give a general criteria under which an approximate ASD connection $A\in\mathcal{A}_{P}$ could deform into an other approximate ASD connection $A_{\infty}$ which satisfies
\begin{equation} \label{E3.1}
\La_{\w}F_{A_{\infty}}=0.
\end{equation}
One also can see \cite[Section 4.2]{Huang2019}. Let $A$ be a connection on a principal $G$-bundle over $X$. The equation (\ref{E3.1}) for a second  connection $A+a$, where $a\in\Om^{1}(X,\mathfrak{g}_{P})$ is a bundle valued $1$-form, could be rewritten to
\begin{equation}\label{E1}
\La_{\w}(d_{A}a+a\wedge a)=-\La_{\w}F_{A}.
\end{equation}
We seek a solution of the equation (\ref{E1}) in the form 
$$a=d_{A}^{\ast}(s\otimes\w)=\sqrt{-1}(\pa_{A}s-\bar{\pa}_{A}s)$$
where $s\in\Om^{0}(X,\mathfrak{g}_{P})$ is a bundle value $0$-form. Then Equation (\ref{E1}) becomes the second order equation:
\begin{equation}\label{E2}
-d_{A}^{\ast}d_{A}s+\La_{\w}(d_{A}s\wedge d_{A}s)=-\La_{\w}F_{A}.
\end{equation}
For convenience, we define a map
$$B(u,v):=\frac{1}{2}\La_{\w}[d_{A}u\wedge d_{A}v].$$
It's easy to check, we have the pointwise bound:
$$|B(u,v)|\leq C|\na_{A}u||\na_{A}v|,$$
where $C$ is a uniform positive constant.

We would like to prove that if $\La_{\w}F_{A}$ is small in an appropriate sense, there is a small solution $s$ to equation (\ref{E2}).
\begin{theorem}\label{T6}(\cite[Theorem 4.8]{Huang2019})
Let $X$ be a compact, K\"{a}hler surface with a K\"{a}hler metric $g$, $P$ be a principal $G$-bundle over $X$ with $G$ being a compact, semisimple Lie group. There are positive constant $\la=\la(g,P)$ and $\varepsilon=\varepsilon(\la,g,P)$ with following significance. If the curvature $F_{A}$ of a connection $A$ on $P$ obeying
\begin{equation*}
\begin{split}
&\|\La_{\w}F_{A}\|_{L^{2}(X)}\leq\varepsilon,\\
&\la(A)\geq\la,\\
\end{split}
\end{equation*} 
where $\la(A)$ is as in (\ref{E3}), then there is a section $s\in\Ga(\mathfrak{g}_{P})$ such that the connection $$A_{\infty}:=A+\sqrt{-1}(\pa_{A}s-\bar{\pa}_{A}s)$$ satisfies\\
(1) $\La_{\w}F_{A_{\infty}}=0$\\
(2) $\|s\|_{L^{2}_{2}(X)}\leq C\|\La_{\w}F_{A}\|_{L^{2}(X)}$,\\
where $C=C(\la,g)\in[1,\infty)$ is a positive constant. Furthermore, let $p\in(2,\infty)$, 
$$\|F^{0,2}_{A_{\infty}}-F_{A}^{0,2}\|_{L^{2}(X)}\leq C(\|\La_{\w}F_{A}\|_{L^{2}(X)}+\|F_{A}^{0,2}\|_{L^{p}(X)})\|\La_{\w}F_{A}\|_{L^{2}(X)},$$
for a positive constant $C=C(\la,g,p)$.
\end{theorem}
Now, we begin to prove Theorem \ref{T6}, the proof of above theorem is base on Taubes' ideas \cite{Taubes1982} and \cite{Donaldson1993}. At first, suppose $s$ and $f$ are sections of $\Ga(\mathfrak{g}_{P)}$ with
\begin{equation}\label{E8}
d^{\ast}_{A}d_{A}s=f,\ i.e.,\ \na_{A}^{\ast}\na_{A}s=f,
\end{equation}
the first observation is
\begin{lemma}\label{L3}(\cite[Lemma 4.9]{Huang2019})
If $\la(A)\geq\la>0$, then there exists a unique $C^{\infty}$ solution to equation (\ref{E8}). Furthermore, we have
\begin{equation*}
\begin{split}
&\|s\|_{L^{2}_{2}(X)}\leq c\|f\|_{L^{2}(X)},\\
&\|B(s,s)\|_{L^{2}(X)}\leq c\|f\|^{2}_{L^{2}(X)},\\
\end{split}
\end{equation*}
where $c=c(\la,g)$ is a positive constant.
\end{lemma} 
\begin{proof}
Following the estimate on Lemma \ref{L1}, we have
$$\|s\|_{L^{2}_{2}(X)}\leq \|\na_{A}^{\ast}\na_{A}s\|_{L^{2}(X)}\leq c\|f\|_{L^{2}(X)},$$
for a positive constant $c=c(\la,g,P)$.	By the Sobolev inequality in four dimension,
\begin{equation*}
\|B(s,s)\|_{L^{2}(X)}\leq C\|\na_{A}s\|^{2}_{L^{4}(X)}\leq C\|\na_{A}s\|^{2}_{L^{2}_{1}(X)}\leq C\|s\|^{2}_{L^{2}_{2}(X)},
\end{equation*}
for a positive constant $C=C(\la,g,P)$.
\end{proof}
\begin{lemma}(\cite[Lemma 4.10]{Huang2019})
If $d_{A}^{\ast}d_{A}s_{1}=f_{1}$, $d_{A}^{\ast}d_{A}s_{2}=f_{2}$, then
\begin{equation*}
\|B(s_{1},s_{2})\|_{L^{2}(X)}\leq c\|f_{1}\|_{L^{2}(X)}\|f_{2}\|_{L^{2}(X)}.
\end{equation*}
\end{lemma}
We will prove the existence of a solution of (\ref{E2}) by the contraction mapping principle. We write $s=(d_{A}^{\ast}d_{A})^{-1}f$ and (\ref{E2}) becomes an equation for $f$ of the from 
\begin{equation}\label{E10}
f-S(f,f)=\La_{\w}F_{A},
\end{equation}
where  $S(f,g):=B((d_{A}^{\ast}d_{A})^{-1}f,(d_{A}^{\ast}d_{A})^{-1}g)$. Following Lemma \ref{L3},
\begin{equation*}
\begin{split}
\|S(f_{1},f_{1})-S(f_{2},f_{2})\|_{L^{2}(X)}&=\|S(f_{1}+f_{2},f_{1}-f_{2})\|_{L^{2}(X)}\\
&\leq c\|f_{1}+f_{2}\|_{L^{2}(X)}\|f_{1}-f_{2}\|_{L^{2}(X)}.\\
\end{split}
\end{equation*}
We denote $g_{k}=f_{k}-f_{k-1}$ and $g_{1}=f_{1}$, then
$$g_{1}=\La_{\w}F_{A},\  g_{2}=S(g_{1},g_{1})$$
and $$g_{k}=S(\sum_{i=1}^{k-1}g_{i},\sum_{i=1}^{k-1}g_{i})-S(\sum_{i=1}^{k-2}g_{i},\sum_{i=1}^{k-2}g_{i}),\ \forall\ k\geq 3.$$
It is easy to show that, under the assumption of $\La_{\w}F_{A}$, the sequence $f_{k}$ defined by
$$f_{k}=S(f_{k-1},f_{k-1})+\La_{\w}F_{A},$$
starting with $f_{1}=\La_{\w}F_{A}$, is Cauchy with respect to $L^{2}$, and so converges to a limit $f$ in the completion of $\Ga(\mathfrak{g}_{P})$ under $L^{2}$.
\begin{proposition}(\cite[Proposition 4.11]{Huang2019})
There are positive constants $\varepsilon\in(0,1)$ and $C\in(1,\infty)$ with following significance. If $$\|\La_{\w}F_{A}\|_{L^{2}(X)}\leq\varepsilon,$$
then each $g_{k}$ exists and is $C^{\infty}$. Further for each $k\geq1$, we have
\begin{equation}\label{E9}
\|g_{k}\|_{L^{2}(X)}\leq C^{k-1}\|\La_{\w}F_{A}\|^{k}_{L^{2}(X)}.
\end{equation}
\end{proposition}
\begin{proof}
The proof is by induction on the integer $k$. The induction begins with $k=1$, one can see $g_{1}=\La_{\w}F_{A}$. The induction proof if completed by demonstrating that if (\ref{E9}) is satisfied for $j<k$, then it also satisfied for $j=k$. Indeed, since 
\begin{equation}\nonumber
\begin{split}
\|S(\sum_{i=1}^{k-1}g_{i},\sum_{i=1}^{k-1}g_{i})-S(\sum_{i=1}^{k-2}g_{i},\sum_{i=1}^{k-2}g_{i})\|_{L^{2}(X)}&\leq c\|\sum_{i=1}^{k-1}g_{i}+\sum_{i=1}^{k-2}g_{i}\|_{L^{2}(X)}\|g_{k-1}\|_{L^{2}(X)},\\
&\leq 2c\sum\|g_{i}\|_{L^{2}(X)}\|g_{k-1}\|_{L^{2}(X)},\\
&\leq \frac{2cC^{k-2}\|\La_{w}F_{A}\|^{k}_{L^{2}(X)}}{1-C\|\La_{\w}F_{A}\|_{L^{2}(X)}}.\\
\end{split}
\end{equation}
Now, we provide  $\varepsilon$ sufficiently small and $C$ sufficiently large to ensure $\|\La_{\w}F_{A}\|_{L^{2}(X)}\leq C^{-2}(C-2c)$, i.e., $\frac{2c}{1-C\|\La_{w}F_{A}\|_{L^{2}(X)}}\leq C$,  hence we complete the proof of this Proposition.
\end{proof}
\begin{proof}[\textbf{Proof of Theorem \ref{T6}}] The sequence $g_{k}$ is Cauchy in $L^{2}$, the limit $f:=\lim_{i\rightarrow}f_{k}$ is a solution to (\ref{E10}). Following  Lemma \ref{L3}, we have
\begin{equation*}
\|s\|_{L^{2}_{2}(X)}\leq c\|f\|_{L^{2}(X)}\leq c\sum_{k=1}^{\infty}\|g_{k}\|_{L^{2}(X)}\leq \frac{c\|\La_{\w}F_{A}\|_{L^{2}(X)}}{1-C\|\La_{\w}F_{A}\|_{L^{2}(X)}},
\end{equation*}
for a positive constant $c$. We provide $\varepsilon$ and $C$ to ensure $C\varepsilon\leq\frac{1}{2}$, hence 
$$\|s\|_{L^{2}_{2}(X)}\leq 2c\|\La_{\w}F_{A}\|_{L^{2}(X)}.$$
We denote $A_{\infty}:=A+\sqrt{-1}(\pa_{A}s-\bar{\pa}_{A}s)$, and $r\in(2,\infty)$ defined by $1/r=1/2-1/p$, then 
\begin{equation*}
\begin{split}
\|F^{0,2}_{A_{\infty}}-F_{A}^{0,2}\|_{L^{2}(X)}&=\|-\sqrt{-1}\bar{\pa}_{A}\bar{\pa}_{A}s-\bar{\pa}_{A}s\wedge\bar{\pa}_{A}s\|_{L^{2}(X)}\\
&=\|-\sqrt{-1}[F_{A}^{0,2},s]-\bar{\pa}_{A}s\wedge\bar{\pa}_{A}s\|_{L^{2}(X)}\\
&\leq 2\|\bar{\pa}_{A}s\|^{2}_{L^{4}(X)}+2\|F_{A}^{0,2}\|_{L^{p}(X)}\|s\|_{L^{r}(X)}\\
&\leq c\|\na_{A}s\|^{2}_{L^{2}_{1}(X)}+c\|F_{A}^{0,2}\|_{L^{p}(X)}\|s\|_{L^{2}_{2}(X)}\\
&\leq c(\|\La_{\w}F_{A}\|_{L^{2}(X)}+\|F_{A}^{0,2}\|_{L^{p}(X)})\|\La_{\w}F_{A}\|_{L^{2}(X)}.\\
\end{split}
\end{equation*}
where $c$ is a positive constant. We complete the proof of Theorem \ref{T6}.
\end{proof}
\begin{corollary}\label{C4}
Let $X$ be a compact, simply-connected, K\"{a}hler surface with a K\"{a}hler metric $g$, that is $generic$ in the sense of Definition \ref{D1}, $P$ be a $SO(3)$-bundle over $X$. Suppose that the second Stiefel-Whitney class, $\w_{2}(P)\in H^{2}(X;\mathbb{Z}/2\mathbb{Z})$, is non-trivial, then there is a positive constant $\varepsilon=\varepsilon(g,P)$ with following significance. If the curvature $F_{A}$ of a connection $A$ on $P$ obeying
$$\|F_{A}^{+}\|_{L^{2}(X)}\leq\varepsilon,$$
then  there is a section $s\in\Ga(\mathfrak{g}_{P})$ such that the connection $A_{\infty}:=A+\sqrt{-1}(\pa_{A}s-\bar{\pa}_{A}s)$ satisfies\\
(1) $\La_{\w}F_{A_{\infty}}=0$\\
(2) $\|s\|_{L^{2}_{2}(X)}\leq C\|\La_{\w}F_{A}\|_{L^{2}(X)}$,\\
where $C=C(g,P)$ is a positive constant. 
\end{corollary}
\subsection{Yang-Mills connection with harmonic curvature}
Suppose that an integrable connecion $A\mathcal{A}_{P}^{1,1}$ on a holomorphic bundle over a K\"{a}hler surface is Yang-Mills, then $\La_{\w}F_{A}$ is parallel, that is equivalent to $F_{A}^{0,2}$ being harmonic with respect to Laplacian operator $\De_{\bar{\pa}_{A}}$. For a general case, we introduce the definition of a connection with harmonic curvature, See \cite[p. 96]{Itoh2}. 
\begin{definition}
A connection $A$ on a compact K\"{a}hler surface is said to be with a harmonic curvature if $(0,2)$-part of curvature is harmonic, i.e., $\bar{\pa}_{A}^{\ast}F_{A}^{0,2}=0$.
\end{definition}
\begin{lemma}\label{L5}
Let $X$ be a compact K\"{a}hler surface, $P$ be a $G$-bundle over $X$ with $G$ being a compact, semisimple Lie group, $A$ be a Yang-Mills connection on $P$. There are positive constant $\la=\la(g,P)$ and $\varepsilon=\varepsilon(\la,g,P)$ with following significance. If the curvature $F_{A}$ of the connection $A$ obeying
\begin{equation*}
\begin{split}
&\|F_{A}^{+}\|_{L^{2}(X)}\leq \varepsilon,\\
&\la(A)\geq\la,\\
\end{split}
\end{equation*} 
where $\la(A)$ is as in (\ref{E3}), then the curvature is harmonic and $\La_{\w}F_{A}=0$.
\end{lemma}
\begin{proof}
For a suitable constant $\varepsilon$, from Theorem \ref{T6}, there exist a connection $A_{\infty}$ such that
$$\|A-A_{\infty}\|_{L^{2}_{1}(X)}\leq c\|\La_{\w}F_{A}\|_{L^{2}(X)}$$
and $\La_{\w}F_{A_{\infty}}=0$. We apply the Weizenb\"{o}ck formula (\ref{E14}) to $F_{A}^{0,2}$, 
$$\|\bar{\pa}_{A_{\infty}}F_{A}^{0,2}\|^{2}_{L^{2}(X)}+\|\bar{\pa}_{A_{\infty}}^{\ast}F_{A}^{0,2}\|^{2}_{L^{2}(X)}=\|\na_{A_{\infty}}F_{A}^{0,2}\|^{2}_{L^{2}(X)}+\int_{X}2S|F_{A}^{0,2}|^{2}dvol_{g}.$$
We observe that $\bar{\pa}_{A_{\infty}}F_{A}^{0,2}=0$ and 
\begin{equation*}
\begin{split}
\|\bar{\pa}^{\ast}_{A_{\infty}}F_{A}^{0,2}\|^{2}_{L^{2}(X)}
&\leq c\|\{A-A_{\infty},F_{A}^{0,2}\}\|^{2}_{L^{2}(X)}+\|\bar{\pa}_{A}^{\ast}F_{A}^{0,2}\|^{2}_{L^{2}(X)}\\
&\leq c\|\{A-A_{\infty},F_{A}^{0,2}\}\|^{2}_{L^{2}(X)}+\frac{1}{4}\|\bar{\pa}_{A}\La_{\w}F_{A}\|^{2}_{L^{2}(X)},\\
&\leq c\|A-A_{\infty}\|^{2}_{L^{4}(X)}\|F_{A}^{0,2}\|^{2}_{L^{4}(X)}+\frac{1}{4}\|\bar{\pa}_{A}\La_{\w}F_{A}\|^{2}_{L^{2}(X)},\\
&\leq c\|A-A_{\infty}\|^{2}_{L^{2}_{1}(X)}\|F_{A}^{0,2}\|^{2}_{L^{2}(X)}+\frac{1}{4}\|\bar{\pa}_{A}\La_{\w}F_{A}\|^{2}_{L^{2}(X)},\\
&\leq c\|\La_{\w}F_{A}\|^{2}_{L^{2}(X)}\|F_{A}^{0,2}\|^{2}_{L^{2}(X)}+\frac{1}{4}\|\bar{\pa}_{A}\La_{\w}F_{A}\|^{2}_{L^{2}(X)}.\\
\end{split}
\end{equation*}
Here we use the estimates on Proposition \ref{P4} and Theorem \ref{T6} and Sobolev embedding $L^{2}_{1}\hookrightarrow L^{4}$. 
Combining the preceding inequalities with integrable identity (\ref{E12}) on Proposition \ref{P5}, gives
\begin{equation*}
\begin{split}
\frac{3}{4}\|\bar{\pa}_{A}\La_{\w}F_{A}\|^{2}_{L^{2}(X)}&= \|\na_{A}F_{A}^{0,2}\|^{2}_{L^{2}(X)}+\int_{X}2S|F_{A}^{0,2}|^{2}dvol_{g}\\
&\leq \|\na_{A_{\infty}}F^{0,2}_{A}\|^{2}_{L^{2}(X)}+\int_{X}2S|F_{A}^{0,2}|^{2}dvol_{g}+\|\{A-A_{\infty},F_{A}^{0,2}\}\|^{2}_{L^{2}(X)}\\
&\leq c\|\La_{\w}F_{A}\|^{2}_{L^{2}(X)}\|F_{A}^{0,2}\|^{2}_{L^{2}(X)}+\frac{1}{4}\|\bar{\pa}_{A}\La_{\w}F_{A}\|^{2}_{L^{2}(X)}.\\
\end{split}
\end{equation*}
for a positive constant $c=c(\la,g)$. Thus, we have
\begin{equation}\label{E15}
\|\bar{\pa}_{A}\La_{\w}F_{A}\|^{2}_{L^{2}(X)}\leq c\|\La_{\w}F_{A}\|^{2}_{L^{2}(X)}\|F_{A}^{0,2}\|^{2}_{L^{2}(X)}.
\end{equation}
We apply  Weitzenb\"{o}ck formula to  $\La_{\w}F_{A}$, See \cite[Lemma 6.1]{DK} .
$$\bar{\pa}^{\ast}_{A}\bar{\pa}_{A}\La_{\w}F_{A}=\frac{1}{2}\na_{A}^{\ast}\na_{A}\La_{\w}F_{A}+[\sqrt{-1}\La_{\w}F_{A},\La_{\w}F_{A}],$$
thus
$$\|\na_{A}\La_{\w}F_{A}\|^{2}_{L^{2}(X)}=2\|\bar{\pa}_{A}\La_{\w}F_{A}\|_{L^{2}(X)}^{2}.$$
Combining above identity with estimate (\ref{E15}) yields,
$$\|\La_{\w}F_{A}\|^{2}_{L^{2}(X)}\leq c\|\na_{A}\La_{\w}F_{A}\|^{2}_{L^{2}(X)}\leq c\|\La_{\w}F_{A}\|^{2}_{L^{2}(X)}\|F_{A}^{0,2}\|^{2}_{L^{2}(X)},$$
where $c=c(\la,g)$ is a positive constant. Provide $c\|F_{A}^{0,2}\|^{2}_{L^{2}(X)}\leq\frac{1}{2}$, thus $\La_{\w}F_{A}\equiv0$.
\end{proof}
Following the Lemma \ref{L5} and the eigenvalue $\la(A)$ has a uniform positive lower bounded under the hypothesis of K\"{a}hler metric $g$ is generic and the curvature $F_{A}$ of the connection $[A]$ obeys $\|F_{A}^{+}\|_{L^{2}(X)}\leq\varepsilon$ for a small enough constant $\varepsilon$, then we have 
\begin{corollary}\label{C2}
Let $X$ be a compact, simply-connected, K\"{a}hler surface with a K\"{a}hler metric $g$, that is $generic$ in the sense of Definition \ref{D1}, $P$ be a $SO(3)$-bundle over $X$. Suppose that the second Stiefel-Whitney class, $\w_{2}(P)\in H^{2}(X;\mathbb{Z}/2\mathbb{Z})$, is non-trivial, then there is a positive constant $\varepsilon=\varepsilon(g,P)$ with following significance. If the curvature $F_{A}$ of a Yang-Mills connection $A$ on $P$ obeying
$$\|F_{A}^{+}\|_{L^{2}(X)}\leq\varepsilon,$$
then the curvature is harmonic and $\La_{\w}F_{A}=0$.
\end{corollary}
\subsection{A vanishing theorem}
Let $(X,\w)$ be a compact K\"{a}hler surface. Given an orthonormal coframe $\{e_{0},e_{1},e_{2},e_{3}\}$ on $X$ for which $\w=e^{01}+e^{23}$, where $e^{ij}=e^{i}\wedge e^{j}$. We define $dz^{1}=e^{0}+\rm{i}e^{1}$, $dz^{2}=e^{2}+\rm{i}e^{3}$ and $d\bar{z}^{1}=e^{0}-\rm{i}e^{1}$, $d\bar{z}^{2}=e^{2}-\rm{i}e^{3}$, so that $\w=\frac{\rm{i}}{2}(dz^{1}\wedge d\bar{z}^{1}+dz^{2}\wedge d\bar{z}^{2})$. 
\begin{proposition}
Let $A$ be a connection on a principal $SU(2)$ or $SO(3)$ bundle over a compact K\"{a}hler surface. If the curvature $F_{A}$ of connection $A$ is harmonic, then $F_{A}^{0,2}$ has at most rank one.
\end{proposition}
\begin{proof}
Since $\bar{\pa}_{A}^{\ast}F_{A}^{0,2}=0$, we have
\begin{equation}\label{E3.7}
0=\bar{\pa}^{\ast}_{A}\bar{\pa}^{\ast}_{A}F_{A}^{0,2}=-\ast[F_{A}^{0,2}\wedge\ast F_{A}^{0,2}].
\end{equation} 
In an orthonormal coframe, we can written $F_{A}^{0,2}$ as 
$$F_{A}^{0,2}=(B_{1}+{\rm{i}}B_{2})d\bar{z}^{1}\wedge d\bar{z}^{2},$$
where $B_{1},B_{2}$ take value in Lie algebra $\mathfrak{su}(2)$ or $\mathfrak{so}(3)$. Thus 
$$\ast F_{A}^{0,2}=(-B_{1}+{\rm{i}}B_{2})dz^{1}\wedge dz^{2}.$$
Following Equation (\ref{E3.7}), we obtain that
\begin{equation}\label{E3.8}
0=[B_{1},B_{2}].
\end{equation}
Thus $F_{A}^{0,2}$ has most rank one. For the details of the calculation, one also can see \cite{Mares} Chapter 4. 
\end{proof}
We define $\b$ as follows, if
$$B=B_{1}(e^{01}+e^{23})+B_{2}(e^{02}+e^{31})+B_{3}(e^{03}+e^{12}),$$
then
$$\b:=\frac{1}{2}(B_{2}-{\rm{i}}B_{3})dz^{1}\wedge dz^{2},\ \b^{\ast}:=-\frac{1}{2}(B_{2}+{\rm{i}}B_{3})d\bar{z}^{1}\wedge d\bar{z}^{2}.$$
It follows that $B:=B_{1}\w+\b-\b^{\ast}$. We define a bilinear map $$[\bullet.\bullet]:\Om^{2,+}(X,\mathfrak{g}_{P})\otimes\Om^{2,+}(X,\mathfrak{g}_{P})\rightarrow\Om^{2,+}(X,\mathfrak{g}_{P})$$
by $\frac{1}{2}[\cdot,\cdot]_{\Om^{2,+}}\otimes[\cdot,\cdot]_{\mathfrak{g}_{P}}$, see \cite{Mares} Section B.4.  In a direct calculate, see \cite{Mares} Section 7.1, 
$$-\frac{1}{4}[B.B]=[B_{2},B_{3}](e^{01}+e^{23})+[B_{3},B_{1}](e^{02}+e^{31})+[B_{1},B_{2}](e^{03}+e^{12}).$$
\begin{proposition}
Let $G$ be a compact, semisimple  Lie group, $A$ be a connection on a principal $G$-bundle over a compact K\"{a}hler surface. If the curvature of the connection $A$ satisfies $\La_{\w}F_{A}=0$ and $\bar{\pa}_{A}^{\ast}F_{A}^{0,2}=0$, then $[F_{A}^{+}. F_{A}^{+}]=0$.	
\end{proposition}
\begin{proof}
We can written $F_{A}^{+}$ as $F_{A}^{+}:=F_{A}^{0,2}+F_{A}^{2,0}+\frac{1}{2}\La_{\w}F_{A}\otimes\w$. By the hypothesis of curvature and Equation (\ref{E3.8}), we then have $[F_{A}^{+}. F_{A}^{+}]=0$.
\end{proof}
Before the prove of vanishing theorem \ref{T2}, we should recall a useful lemma proved by Donaldson \cite{DK} Lemma 4.3.21.
\begin{lemma}\label{L4}
If $A$ is an irreducible $SU(2)$ or $SO(3)$ ASD connection on a bundle $P$ over a simply connected four-manifold $X$, then the restriction of $A$ to any non-empty open set in $X$ is also irreducible.
\end{lemma}
We recall the following simple but powerful corollary of unique continuation for ASD connections which proved in \cite{Mares} Theorem 4.2.1. For the convenience of the readers, we give a proof of this theorem.
\begin{theorem}\label{T3.12}
Let $X$ be a simply-connected, oriented, smooth Riemannian four-manifold, $P$ be a principal $SU(2)$ or $SO(3)$ bundle over $X$ and $A$ be an irreducible ASD connection on $P$. If $B\in\Om^{2,+}(X,\mathfrak{g}_{P})$ satisfies 
$$d_{A}^{\ast}B=0\ and\ [B. B]=0,$$ 
then $B=0$.	
\end{theorem}
\begin{proof}
Let $Z^{c}$ denote the complement of the zero set of $B$. By unique continuation of the elliptic equation $d_{A}^{\ast}B=0$, $Z^{c}$ is eithor empty or dense. On $Z^{c}$ write $B=f\otimes\sigma$ for $\sigma\in\Om^{0}(Z^{c},\mathfrak{g}_{P})$ with $|\sigma|^{2}=1$ and $f\in\Om^{2,+}(Z^{c})$. We compute
$$0=d_{A}^{\ast}B=-\ast d_{A}(f\otimes\sigma)=-\ast(df\otimes\sigma+f\otimes d_{A}\sigma).$$
Taking the inner product with $\sigma$ and using the consequence of $|\sigma|^{2}=1$ that $\langle\sigma,d_{A}\sigma\rangle=0$, we get $df=0$. It follows that $f\otimes d_{A}\sigma=0$. Since $f$ is nowhere zero along $Z^{c}$, we have $d_{A}\sigma=0$ along $Z^{c}$. Therefore, $A$ is reducible along $Z^{c}$. However according to Lemma \ref{L4}, $A$ is irreducible along $Z^{c}$. This is a contradiction unless $Z^{c}$ is empty. Therefore $Z=X$, so $B=0$.
 \end{proof}
\begin{corollary}\label{T2}
Let $X$ be a compact, simply-connected, K\"{a}hler surface, $P$ be a principal $G=SU(2)$ or $SO(3)$ bundle over $X$ and $A$ be an irreducible ASD connection on $P$. If $\phi\in\Om^{0,2}(X,\mathfrak{g}_{P}^{\C})$ satisfies $$[\phi,\ast\phi]=0\ and\ \bar{\pa}_{A}^{\ast}\phi=0,$$ 
then $\phi$ vanish. 
\end{corollary}
\begin{proof}
By the hypothesis of $\phi$, it follows that $B:=\phi-\phi^{\ast}$ satisfies $[B.B]=0$ and $d_{A}^{\ast}B=0$. Following vanishing theorem \ref{T3.12}, we obtain that $B=0$, i.e., $\phi=0$.
\end{proof}
At first, we define a subset of $\Om^{0,2}(X,\mathfrak{g}_{P}^{\C})$ as follow:
\begin{equation}\label{E3.10}
\tilde{\Om}^{0,2}(X,\mathfrak{g}_{P}^{\C})=\{\phi\in\Om^{0,2}(X,\mathfrak{g}_{P}^{\C}): [\phi,\ast\phi]=0\}.
\end{equation}
\begin{definition}\label{D3.13}
The least eigenvalue of $\bar{\pa}_{A}\bar{\pa}_{A}^{\ast}$ on  $L^{2}(\tilde{\Om}^{0,2}(X,\mathfrak{g}_{P}^{\C}))$ is
\begin{equation}
\mu(A):=\inf_{v\in\tilde{\Om}^{0,2}(X,\mathfrak{g}_{P}^{\C})\backslash\{0\}}\frac{\|\bar{\pa}^{\ast}_{A}v\|^{2}}{\|v\|^{2}}.
\end{equation} 
\end{definition}
\begin{proposition}\label{P3.12}
Let $X$ be a simply-connected, compact K\"{a}hler surface with a K\"{a}hler metric $g$, $P$ be a principal $SU(2)$ or $SO(3)$ bundle over $X$. If $A$ is an irreducible anti-self-dual connection on $P$, then $\mu(A)>0$.
\end{proposition}
\begin{proof}
If not, the eigenvalue $\mu(A)=0$. We may then choose a sequence $\{v_{i}\}_{i\in\N}\subset\tilde{\Om}^{0,2}\backslash\{0\}$ such that $$\|\bar{\pa}^{\ast}_{A}v_{i}\|^{2}_{L^{2}(X)}\leq\mu_{i}\|v_{i}\|^{2}_{L^{2}(X)}$$
and
$$\mu_{i}\rightarrow0^{+}\ as\ i\rightarrow\infty.$$
Since $[\frac{v}{\|v\|_{L^{2}}}\wedge\ast\frac{v}{\|v\|_{L^{2}}}]=0$ for $v\in\tilde{\Om}^{0,2}\backslash\{0\}$, we then noting $\|v_{i}\|_{L^{2}(X)}=1$, $\forall i\in\mathbb{N}$. Following the Weizenb\"{o}ck formula, we have
$$\|\na_{A}v_{i}\|^{2}_{L^{2}(X)}=-\langle Sv_{i},v_{i}\rangle_{L^{2}(X)}+\|\bar{\pa}_{A}^{\ast}v_{i}\|^{2}_{L^{2}(X)}.$$
Thus
$$\|v_{i}\|^{2}_{L^{2}_{1}}\leq(C+\la_{i})<\infty,$$
where $C$ is a positive constant only dependence on the metric. Therefore, there exist a subsequence $\Xi\subset N$ such that $\{v_{i}\}_{i\in\Xi}$ weakly convergence to $v_{\infty}$ in $L^{2}_{1}$, we also have
$\bar{\pa}_{A}^{\ast}v_{i}$ converge weakly in $L^{2}$ to a limit $\bar{\pa}_{A}^{\ast}v_{\infty}=0$. On the other hand, $L^{2}_{1}\hookrightarrow L^{p}$, for $2\leq p<4$, we may choose $p=2$, then 
\begin{equation}\nonumber
\begin{split}
\|[v_{\infty}\wedge\ast v_{\infty}]\|_{L^{1}(X)}&=\|[(v_{\infty}-v_{i})\wedge \ast v_{\infty}+v_{i}\wedge\ast(v_{\infty}-v_{i}) ]\|_{L^{1}(X)}\\
&\leq\|v_{i}-v_{\infty}\|_{L^{2}(X)}(\|v\|_{L^{2}(X)}+\|v_{i}\|_{L^{2}(X)})\rightarrow0\ as\ i\rightarrow\infty,\\
\end{split}
\end{equation}
Hence
$$[v_{\infty}\wedge\ast v_{\infty}=0],\ i.e.,\ v_{\infty}\in\tilde{\Om}^{0,2}.$$
Therefore the corollary \ref{T2} implies that $\ker{\bar{\pa}_{A}^{\ast}}|_{\tilde{\Om}^{0,2}(X,\mathfrak{g}_{P}^{\C})}=0$. Thus $v_{\infty}$ vanish. It's contradicting to $\|v_{\infty}\|_{L^{2}(X)}=1$. In particular, the preceding arguments shows that the $\mu(A)>0$.
\end{proof}
\begin{lemma}(\cite{DK} Lemma 7.2.10)\label{L3.15}
There is a universal constant $C$ and for any $N\geq2$, $R>0$, a smooth radial function $\b=\b_{N,R}$ on $\mathbb{R}^{4}$, with
$$0\leq\b(x)\leq1$$
$$\b(x)=\left\{
\begin{aligned}
1&   &|x|\leq R/N \\
0&   &|x|\geq R
\end{aligned}
\right.$$
and
$$\|\na\b\|_{L^{4}}+\|\na^{2}\b\|_{L^{2}}<\frac{C}{\sqrt{\log N}}.$$
Assuming $R<R_{0}$, the same holds for $\b(x-x_{0})$ on any geodesic ball $B_{R}(x_{0})\subset X$.
\end{lemma}
Following the idea in \cite{Feehan}, we can prove that the least eigenvalue of $\bar{\pa}_{A}\bar{\pa}_{A}^{\ast}$ on the space $\tilde{\Om}^{0,2}(X,\mathfrak{g}_{P}^{\C})$ with respect to the connection $A$ is continuity in the sense of $L^{4}_{loc}$.
\begin{proposition}\label{P3.16}
Let $X$ be a compact K\"{a}hler surface, $\Sigma=\{x_{1},x_{2},\ldots,x_{L}\}\subset X$ ($L\in\N^{+}$) and $0<\rho\leq\min_{i\neq j}dist_{g}(x_{i},x_{j})$ and $U\subset X$ be the open subset give by $$U:=X\backslash\bigcup_{l=1}^{L}\bar{B}_{\rho/2}(x_{l}).$$ 
Let $G$ be a compact, semisimple  Lie group, $A_{0}$  be a $C^{\infty}$ connection on a principal $G$-bundle $P_{0}$ over $X$ obeying the curvature bounded
\begin{equation}\label{E3.9}
\|F^{+}_{A_{0}}\|_{L^{2}(X)}\leq\varepsilon
\end{equation} 
where $\varepsilon\in(0,1)$ is a sufficiently small positive constant. Let $P$ be a principal $G$-bundle over $X$ such that there is an isomorphism of principal $G$-bundles, $u:P\upharpoonright X\backslash\Sigma\cong P_{0}\upharpoonright X\backslash\Sigma$, and identify $P\upharpoonright X\backslash\Sigma$ with $P_{0}\upharpoonright X\backslash\Sigma$ using this isomorphism. Then there are positive constants $c=c(\rho,g)\in(0,1]$,  $c\in(1,\infty)$ and $\de\in(0,1]$ with the following significance. Let $A$ be a $C^{\infty}$ connection on $P$ obeying the curvature bounded (\ref{E3.9}) with constant $\varepsilon$ such that  
$$\|A-A_{0}\|_{L^{4}(U)}\leq\de.$$
Then $\mu(A)$ and $\mu(A_{0})$ satisfy
\begin{equation}\label{E3.2}
\mu(A)\leq(1+\eta)\mu(A_{0})+c\big{(}(1+\eta)(C+\de^{2})+(1+\eta^{-1})L\rho^{2}\mu(A)\big{)}(1+\mu(A_{0}))
\end{equation}
and
\begin{equation}\label{E3.3}
\mu(A_{0})\leq(1+\eta)\mu(A)+c\big{(}(1+\eta)(C+\de^{2})+(1+\eta^{-1})L\rho^{2}\mu(A_{0})\big{)}(1+\mu(A))
\end{equation}
where $\eta\in(0,\infty)$ is a positive constant.
\end{proposition}
\begin{proof}
Assume first that $supp(v)\subset U$, write $a:=A-A_{0}$. We then have
$$\big{|}\|\bar{\pa}^{\ast}_{A}v\|^{2}-\|\bar{\pa}^{\ast}_{A_{0}}v\|^{2}\big{|}\leq 2\|a\|^{2}_{L^{4}}\|v\|^{2}_{L^{4}}.$$
On the other hand, if $supp(v)\subset\bigcup_{l=1}^{L}\bar{B}_{\rho/2}(x_{l})$, then
$$\|v\|^{2}_{L^{2}}\leq cL\rho^{2}\|v\|^{2}_{L^{4}}.$$
Let $\psi=\sum\b_{N,\rho}(x-x_{i})$ be a sum of the logarithmic cut-off functions of Lemma \ref{L3.15}, and $\tilde{\psi}=1-\psi$. We now choose $v\in\tilde{\Om}^{0,2}$ with $\|v\|_{L^{2}(X)}=1$. At last, we observe that
$$[\tilde{\psi}v\wedge\ast\tilde{\psi}v]=0,\ i.e.,\ \tilde{\psi}v\in\tilde{\Om}^{0,2}.$$ 
By the definition of $\mu(A)$, we have
$$\mu(A)\|\tilde{\psi}v\|^{2}\leq\|\bar{\pa}^{\ast}_{A}(\tilde{\psi}v)\|^{2}.$$
Following the Weitzenb\"{o}ck formula for $v\in\Om^{0,2}(X,\mathfrak{g}_{P})$, we have
\begin{equation*}
\begin{split}
\|\na_{A}v\|^{2}_{L^{2}(X)}&\leq C\|v\|^{2}_{L^{2}(X)}+\|\bar{\pa}_{A}^{\ast}v\|^{2}_{L^{2}(X)}+\|F^{+}_{A}\|_{L^{2}(X)}\|v\|^{2}_{L^{4}(X)}\\
&\leq C\|v\|^{2}_{L^{2}(X)}+\|\bar{\pa}_{A}^{\ast}v\|^{2}_{L^{2}(X)}+C\|F^{+}_{A}\|_{L^{2}(X)}(\|v\|^{2}_{L^{2}(X)}+\|\na_{A}v\|^{2}_{L^{2}(X)}),\\
\end{split}
\end{equation*}
where $C$ is a positive constant dependence on $X,g$. Provided $C\|F^{+}_{A}\|_{L^{2}(X)}\leq\frac{1}{2}$, we then have a priori estimate for $v\in\Om^{0,2}(X,\mathfrak{g}_{P}^{\C})$,
$$\|\na_{A}v\|^{2}_{L^{2}(X)}\leq C(\|v\|^{2}_{L^{2}(X)}+\|\bar{\pa}_{A}^{\ast}v\|^{2}_{L^{2}(X)}).$$
Combining the above observations, we have
\begin{equation}\label{E3.11}
\begin{split}
\mu(A)\|v\|^{2}_{L^{2}(X)}&\leq\mu(A)(\|\psi v\|^{2}_{L^{2}(X)}+\|\tilde{\psi}v\|^{2}_{L^{2}(X)}+2\langle\psi v,\tilde{\psi}v\rangle_{L^{2}(X)})\\
&\leq\mu(A)(1+\eta^{-1})\|\psi v\|^{2}_{L^{2}(X)}+\mu(A)(1+\eta)\|\tilde{\psi}v\|^{2}_{L^{2}(X)},\\
\end{split}
\end{equation}
where $\eta\in(0,\infty)$ is a positive constant.\\
For the first term on right-hand of (\ref{E3.11}), 
\begin{equation*}
\mu(A)(1+\eta^{-1})\|\psi v\|^{2}_{L^{2}(X)}\leq c(1+\eta^{-1})L\rho^{2}\mu(A)\|v\|^{2}_{L^{4}(X)},
\end{equation*}
for some positive constant $c=c(g)$.\\
For the second term on right-hand of (\ref{E3.7}),
\begin{equation*}
\begin{split}
\mu(A)(1+\eta)\|\tilde{\psi}v\|^{2}_{L^{2}(X)})&\leq (1+\eta)\|\bar{\pa}^{\ast}_{A}(\tilde{\psi}v)\|^{2}_{L^{2}(X)}\\
&\leq(1+\eta)(\|\bar{\pa}^{\ast}_{A_{0}}(\tilde{\psi}v)\|^{2}_{L^{2}(X)}+2\|a\|^{2}_{L^{4}(U)}\|v\|^{2}_{L^{4}(X)})\\
&\leq(1+\eta)(\|\tilde{\psi}\bar{\pa}^{\ast}_{A_{0}}v\|^{2}_{L^{2}(X)}+\|(\na\tilde{\psi})v\|^{2}_{L^{2}(X)}+2\|a\|^{2}_{L^{4}(U)}\|v\|^{2}_{L^{4}(X)})\\
&\leq(1+\eta)\big{(}\|\bar{\pa}^{\ast}_{A_{0}}v\|^{2}_{L^{2}(X)}
+(\|\na\tilde{\psi}\|^{2}_{L^{4}(X)}+2\|a\|^{2}_{L^{4}(U)})\big{)}\|v\|^{2}_{L^{4}(X)}\\
\end{split}
\end{equation*}
Combining the preceding inequalities,
\begin{equation}\nonumber
\begin{split}
\mu(A)&\leq(1+\eta)\|\bar{\pa}^{\ast}_{A_{0}}v\|^{2}_{L^{2}(X)}\\
&+\big{(}(1+\eta)(\|\na\tilde{\psi}\|^{2}_{L^{4}(X)}+\|a\|^{2}_{L^{4}(U)})+c(1+\eta^{-1})L\rho^{2}\mu(A)\big{)}\|v\|^{2}_{L^{4}(X)}\\
&\leq(1+\eta)\|\bar{\pa}^{\ast}_{A_{0}}v\|^{2}_{L^{2}(X)}\\
&+c\big{(}(1+\eta)(\|\na\tilde{\psi}\|^{2}_{L^{4}(X)}+\|a\|^{2}_{L^{4}(U)})+c(1+\eta^{-1})L\rho^{2}\mu(A)\big{)}(1+\|\bar{\pa}^{\ast}_{A_{0}}v\|^{2}_{L^{2}(X)})\\
\end{split}
\end{equation}
In the space $\tilde{\Om}^{0,2}(X,\mathfrak{g}_{P}^{\C})$, we can choose a sequence $v_{\tilde{\varepsilon}}\in\tilde{\Om}^{0,2}$, $\tilde{\varepsilon}\rightarrow0$, such that $$\|\bar{\pa}^{\ast}_{A_{0}}v_{\tilde{\varepsilon}}\|^{2}_{L^{2}(X)}\leq(\mu(A_{0})+\tilde{\varepsilon})\|v_{\tilde{\varepsilon}}\|^{2}\ and \ \|v_{\tilde{\varepsilon}}\|^{2}=1.$$
Therefore,
\begin{equation}\nonumber
\begin{split}
\mu(A)&\leq c(1+\eta)(\|\na\tilde{\psi}\|^{2}_{L^{4}(X)}+\|a\|^{2}_{L^{4}(U)})+c(1+\eta^{-1})L\rho^{2}\mu(A)\big{)}(1+\mu(A_{0})+\tilde{\varepsilon})\\
&+(1+\eta)(\mu(A_{0})+\tilde{\varepsilon}),\\
\end{split}
\end{equation}
Let $\tilde{\varepsilon}\rightarrow0^{+}$, we then have
$$\mu(A)\leq (1+\eta)\mu(A_{0})+c(1+\eta)\big{(}(\|\na\tilde{\psi}\|^{2}_{L^{4}(X)}+\|a\|^{2}_{L^{4}(U)})+(1+\eta^{-1})L\rho^{2}\mu(A)\big{)}(1+\mu(A_{0})).$$
Since $\|\na\tilde{\psi}\|^{2}_{L^{4}(X)}\leq\frac{C'}{\log N}$ for a uniform constant $C'$, we denote $C=\frac{C'}{\log N}$, (see Lemma \ref{L3.15}), we then have
$$\mu(A)\leq(1+\eta)\mu(A_{0})+c\big{(}(1+\eta)(C+\de^{2})+(1+\eta^{-1})L\rho^{2}\mu(A)\big{)}(1+\mu(A_{0}))$$
Therefore, exchange the roles of $A$ and $A_{0}$ in the preceding derivation yields the inequality (\ref{E3.3}) for $\mu(A)$ and $\mu(A_{0})$.
\end{proof}
We then have the convergence of the least eigenvalue of $\bar{\pa}_{A_{i}}\bar{\pa}^{\ast}_{A_{i}}|_{\tilde{\Om}^{0,2}}$ for a sequence of connections $\{A_{i}\}_{i\in\mathbb{N}}$ converging strongly in $L^{2}_{1,loc}(X\backslash\Sigma)$.
\begin{corollary}\label{C3.9}
Let $G$ be a compact, semisimple  Lie group and $P$ be a principal $G$-bundle over a compact K\"{a}hler surface $X$ and $\{A_{i}\}_{i\in\mathbb{N}}$  a sequence of smooth connections on $P$ that converges strongly in $L^{2}_{1,loc}(X\backslash\Sigma)$, moduli a sequence  $\{u_{i}\}_{i\in\mathbb{N}}: P_{\infty}\upharpoonright X\backslash\Sigma\cong P\upharpoonright X\backslash\Sigma$ of class $L^{3}_{1,loc}(X\backslash\Sigma)$ to a connection $A_{\infty}$ on a principal $G$-bundle $P_{\infty}$ over $X$. Then
$$\lim_{i\rightarrow\infty}\mu(A_{i})=\mu(A_{\infty}),$$
where $\mu(A)$ is as in Definition \ref{D3.13}.
\end{corollary}
\begin{proof}
By the Sobolev embedding $L^{2}_{1}\hookrightarrow L^{4}$ and Kato inequality, we have
$$\|u^{\ast}_{i}(A_{i})-A_{\infty}\|_{L^{4}(U)}\rightarrow0\ strongly\ in\ L^{2}_{1,A_{\infty}}(U,\Om^{1}\otimes \mathfrak{g}_{P_{\infty}})\ as\ i\rightarrow\infty.$$
Hence from the inequalities on Proposition \ref{P3.16}, we have
\begin{equation}\label{E3.5}
\mu(A_{\infty})\leq(1+\eta)\liminf_{i\rightarrow\infty}\mu(A_{i})+c\big{(}(1+\eta)C+(1+\eta^{-1})L\rho^{2}\mu(A_{\infty})\big{)}
(1+\liminf_{i\rightarrow\infty}\mu(A_{i}))
\end{equation}
and
\begin{equation}\label{E3.6}
\limsup_{i\rightarrow\infty}\mu(A_{i})\leq(1+\eta)\mu(A_{\infty})+c\big{(}(1+\eta)C+(1+\eta^{-1})L\rho^{2}\limsup_{i\rightarrow\infty}\mu(A_{i})\big{)}(1+\mu(A_{\infty}))
\end{equation}
The inequalities (\ref{E3.5}) and (\ref{E3.6}) about $\liminf_{i\rightarrow\infty}\mu(A_{i})$ and $\limsup_{i\rightarrow\infty}\mu(A_{i})$ hold for every $\rho\in(0,\rho_{0}]$ and $\eta\in(0,\infty)$. It's easy to see that $C\rightarrow0^{+}$ while $\rho\rightarrow0^{+}$. At first,  let $\rho\rightarrow0^{+}$, we then have
$$\mu(A_{\infty})\leq(1+\eta)\liminf_{i\rightarrow\infty}\mu(A_{i})
\leq(1+\eta)\limsup_{i\rightarrow\infty}\mu(A_{i})\leq(1+\eta)^{2}\mu(A_{\infty}).$$
Next, let $\eta\rightarrow0^{+}$, thus the conclusion follows.
\end{proof}
We then have
\begin{proposition}
Let $G$ be a compact, semisimple  Lie group and $P$ a principal $G$-bundle over a compact K\"{a}hler surface $X$. If $\{A_{i}\}_{i\in\mathbb{N}}$ is a sequence $C^{\infty}$ connection on $P$ and the curvatures $\|F^{+}_{A_{i}}\|_{L^{2}(X)}$ converge to zero as $i\rightarrow\infty$, then there are a finite set of points, $\Sigma=\{x_{1},\ldots,x_{L}\}\subset X$ and a subsequence $\Xi\subset\N$, we also denote by $\{A_{i}\}$, a sequence gauge transformation $\{g_{i}\}_{i\in\Xi}$ such that, $g_{i}(A_{i})\rightarrow A_{\infty}$, a anti-self-dual connection on $P\upharpoonright_{X\setminus\Sigma}$ in the $L^{2}_{1}$-topology over $X\setminus\Sigma$ . Furthermore, we have
$$\lim_{i\rightarrow\infty}\mu(A_{i})=\mu(A_{\infty}),$$
where $\mu(A)$ is as in Definition \ref{D3.13}.
\end{proposition} 
Following Feehan's idea, we then have
\begin{theorem}\label{T3.14}
Let $X$ be a compact, simply-connected, K\"{a}hler surface with a K\"{a}hler metric $g$, $P$ be a principal $G$-bundle with $G$ being $SU(2)$ or $SO(3)$. Suppose that the ASD connections $A\in M(P,g)$ are $strong$ irreducible connections in the sense of Definition \ref{D3}. Then there are positive constants $\mu_{0}=\mu_{0}(P,g)$ and $\varepsilon=\varepsilon(P,g)$ such that
\begin{equation*}
\begin{split}
&\mu(A)\geq\mu_{0},\  \forall [A]\in M(P,g),\\
&\mu(A)\geq\frac{\mu_{0}}{2},\  \forall [A]\in\mathcal{B}_{\varepsilon}(P,g).\\
\end{split}
\end{equation*}
\end{theorem}
\begin{proof}
Combining $\bar{M}(P,g)$ is compact, $\la(A)$ is continuous under the Uhlenbeck topology, $\forall [A]\in\bar{M}(P,g)$ and $A$ is a $strong$ irreducible connection, then any connection $A\in\bar{M}(P,g)$ is $strong$ in the sense of Definition \ref{D3}. Following the Proposition \ref{P3.12} , we then have $\mu(A)\geq\mu_{0}$.
	
Suppose that the constant $\varepsilon\in(0,1]$ does not exist. We may then choose a sequence $\{A_{i}\}_{i\in\N}$ of connection on $P$ such that $\|F^{+}_{A_{i}}\|_{L^{2}(X)}\rightarrow0$ and $\mu(A_{i})\rightarrow0$ as $i\rightarrow\infty$. Since $\lim_{i\rightarrow\infty}\mu(A_{i})=\mu(A_{\infty})$ and $\mu(A_{\infty})\geq\mu_{0}$, then it is contradict to our initial assumption regarding the sequence $\{A_{i}\}_{i\in\N}$. In particular, the preceding argument shows that the desired constant $\varepsilon$ exists.
\end{proof}
\begin{corollary}
Let $X$ be a compact, simply-connected, K\"{a}hler surface with a K\"{a}hler metric $g$, that is $generic$ in the sense of Definition \ref{D1}, $P$ be a $SO(3)$-bundle over $X$. If the second Stiefel-Whitney class, $\w_{2}(P)\in H^{2}(X;\mathbb{Z}/2\mathbb{Z})$, is non-trivial, then there are positive constants $\mu_{0}=\mu_{0}(P,g)$ and $\varepsilon=\varepsilon(P,g)$ such that
\begin{equation*}
\begin{split}
&\mu(A)\geq\mu_{0},\  \forall [A]\in M(P,g),\\
&\mu(A)\geq\frac{\mu_{0}}{2},\  \forall [A]\in\mathcal{B}_{\varepsilon}(P,g).\\
\end{split}
\end{equation*}
\end{corollary}
\begin{proof}[\textbf{Proof of Theorem \ref{T1}}]
For a Yang-Mills connection $A$ on $P$ with $\|F_{A}^{+}\|_{L^{2}(X)}\leq\varepsilon$, where $\varepsilon\in(0,1)$ is as in the hypothesis of Corollary \ref{C2}, then the curvature $F_{A}$ satisfies $\bar{\pa}_{A}^{\ast}F_{A}^{0,2}=0$, $\La_{\w}F_{A}=0$. Following the Definition of $\mu(A)$ and Theorem \ref{T3.14},  we have
$$\frac{\mu_{0}}{2}\|v\|^{2}_{L^{2}(X)}\leq\|\bar{\pa}_{A}^{\ast}v\|^{2}_{L^{2}(X)},\ \forall v\in\tilde{\Om}^{0,2}(X,\mathfrak{g}_{P}^{\C}).$$
where $\mu_{0}$ is the uniform positive lower bound in Theorem \ref{T3.14}. Since $F_{A}^{0,2}$ is  harmonic, $F_{A}^{0,2}=0$ on $X$. Thus we complete this proof.
\end{proof}
Suppose $\phi\in\Om^{0,2}(X,\mathfrak{g}^{\C}_{P})$ takes values in a 1-dimensional subbundle of $\mathfrak{g}_{P}^{\C}$, i.e., suppose that 
$$\phi=f\otimes\sigma,$$ 
where $f$ is a $(0,2)$-form and where $\sigma$ is a section of $\mathfrak{g}_{P}^{\C}$ with $|\sigma|^{2}=1$,
following the idea of Bourguignon-Lawson \cite[Proposition 3.15]{BL}, we then have a useful
\begin{lemma}\label{L2} 
Suppose that the curvature $F_{A}$ of the connection $A$ obeying $\La_{\w}F_{A}=0$. The $\phi$ is harmonic with respect to Laplacian operator $\De_{\bar{\pa}_{A}}$ if only if $f$ is harmonic form and $\sigma$ is parallel aways from the zeros of $f$.
\end{lemma}
\begin{proof}
	The Weizenb\"{o}ck formula for any $\phi\in\Om^{0,2}(X,\mathfrak{g}^{\C}_{P})$ yields, see (\ref{E14})
	$$\na_{A}^{\ast}\na_{A}\phi+ 2S\phi=0.$$
	A direct computation shows that
	$$\na_{A}^{\ast}\na_{A}\phi=(\na^{\ast}\na f)\otimes\sigma-\sum_{j}(\na_{e_{j}}f)\otimes(\na_{A_{e_{j}}}\sigma)+f\otimes(\na_{A}^{\ast}\na_{A}\sigma),$$
	and
	$$\langle S\phi,\phi\rangle=\langle Sf,f\rangle.$$
	Taking the derivative of the condition $|\sigma|^{2}=1$,  we find that $\langle\na_{A}\sigma,\sigma\rangle=0$.  Consequently,
	$$\langle\na_{A}^{\ast}\na_{A}\sigma,\sigma\rangle\equiv-\langle\na_{A_{e_{i},e_{i}}}^{2}\sigma,\sigma\rangle\equiv\sum_{i}\langle\na_{A_{e_{i}}}\sigma,\na_{A_{e_{i}}}\sigma\rangle\equiv|\na_{A}\sigma|^{2}.$$
	We then have
	\begin{equation}\label{E11}
	\langle\na_{A}^{\ast}\na_{A}\phi,\phi\rangle=\langle\na^{\ast}\na f,f\rangle+|f|^{2}|\na_{A}\sigma|^{2}.
	\end{equation}
	The Weitzenb\"{o}ck formula applied to $(0,2)$-forms on $\Om^{0,2}(X)$ , states that 
	$$(d^{\ast}d+dd^{\ast})f=\na^{\ast}\na f+2Sf.$$ 
	Therefore (\ref{E11}) can be rewritten as
	$$\langle\De f,f \rangle+|f|^{2}|\na_{A}\sigma|^{2}=0.$$
	Since $\De\geq 0$ on $X$ we conclude that $\De f=0$ and that $\na_{A}\sigma=0$ away from the zero of $f$. We complete the proof this lemma.
\end{proof}
We apply the proof of Lemma \ref{L4} to ASD connections for group $S^{1}$. If $A$ is an ASD $S^{1}$-connection which is flat in the some ball, then in a radial gauge the connection matrix vanishes over the ball and we deduce that $A$ must be flat everywhere. This is a local argument, so applies to any closed ASD $2$-form. Of course, we have just the same results for self-dual forms. We obtain then:
\begin{corollary}\label{C3}(\cite{DK} Corollary 4.3.23)
Suppose $\w$ is a closed two-form on $X$ which satisfies $\ast\w=\pm\w$. Then if $\w$ vanishes on a non-empty open set in $X$ it is identically zero.
\end{corollary}
Furthermore, if $\w$ is a harmonic two from on $X$, then $\w+\ast\w$ or $\w-\ast\w$ is self-dual or ASD closed two-from. Then if $\w\pm\ast\w$ all vanish on a non-empty open set in $X$ it is identically zero, i.e., $\w$ is identically zero.
\begin{proposition}\label{P3.24}
Suppose that $\w$ is a smooth harmonic $2$-form on a closed, simply-connected, four-manifold $X$. Then if $\w$ vanishes on a non-empty open set in $X$ it is identically zero.
\end{proposition}
We then have
\begin{theorem}
Let $X$ be a compact, simply-connected, K\"{a}hler surface, $P$ be a  principal $G$-bundle over $X$ with $G$ being a compact, semisimple  Lie group, $A$ be an irreducible connection on $P$. If the curvature $F_{A}$ of the connection $A$ obeying $\La_{\w}F_{A}=0$. Then 
the harmonic forms take values in a 1-dimensional subbundle of $\mathfrak{g}_{P}^{\C}$ on $\Om^{0,2}(X,\mathfrak{g}_{P}^{\C})$ with respect to Laplacian operator $\De_{\bar{\pa}_{A}}$ are zero.
\end{theorem}
\begin{proof}
We set $\phi:=f\otimes\sigma$ for any section $\phi$ on $\Om^{0,2}(X,\mathfrak{g}_{P}^{\C})$. We suppose $\phi$ is harmonic with respect to $\De_{\bar{\pa}_{A}}$, then $|f||\na_{A}\sigma|=0$ and $df=d^{\ast}f=0$. We denote a closed set 
$$Z:=\{x\in X:f(x)=0\}\subset X$$ 
by the the zero of harmonic form $f$, i.e. the zero of $\phi$. We then have $\na_{A}\sigma=0$ along $X\backslash Z$, thus the set $$\tilde{Z}:=\{x\in X:\na_{A}\sigma\neq 0\}\subset Z.$$
Since the connection $A$ is irreducible, $\|\na_{A}\sigma\|_{L^{2}(X)}>0$, (See Remark \ref{R2.5}), thus the set $\tilde{Z}$ is non-empty. We could choose a point $p\in\tilde{Z}$ such that $\na_{A}\sigma\neq0$. Then there is a  constant $\rho$ such that the geodesic ball $B_{\rho}(p)\subset\tilde{Z}\subset Z$. Since $f$ is harmonic $(0,2)$-from and $f(x)=0$ for any $x\in B_{\rho}(p)$, $\ast f$ is a harmonic $(2,0)$-from and $\ast f$ also vanishes on $B_{\rho}(p)$.
Following Proposition \ref{P3.24}, $f$ is identically zero. Therefore $Z=X$, so $\phi$ is identically zero.
\end{proof}
\section*{Acknowledgment}
We would like to thank Feehan for kind comments regarding his article \cite{Feehan}. This work was partially supported by Nature Science Foundation of China No. 11801539.

\end{document}